\documentclass[12pt, reqno]{amsart}
\usepackage{amsmath, amsthm, amscd, amsfonts, amssymb, graphicx, color}
\usepackage[bookmarksnumbered, colorlinks, plainpages]{hyperref}

\newtheorem{theorem}{Theorem}[section]
\newtheorem{lemma}[theorem]{Lemma}

\newtheorem{corollary}[theorem]{Corollary}
\theoremstyle{definition}

\numberwithin{equation}{section}

\begin{document}
\setcounter{page}{1}

\title[Jordan Derivation On Generalized Triangular Matrix Rings]{Jordan Derivation On Generalized Triangular Matrix Rings}

\author[P. Danchev]{P. Danchev$^{*}$}
\author[A. Fatehi]{A. Fatehi}
\author[M. Zahiri]{M. Zahiri}
\author[S. Zahiri]{S. Zahiri}

\address{Peter Danchev, Institute of Mathematics and Informatics, Bulgarian Academy of Sciences, 1113 Sofia, Bulgaria}
\email{\textcolor[rgb]{0.00,0.00,0.84}{pvdanchev@yahoo.com; danchev@math.bas.bg}}
\address{Ayda Fatehi, Department of Mathematics, Tarbiat Modares University, 14115-111 Tehran Jalal AleAhmad Nasr, Iran}
\email{\textcolor[rgb]{0.00,0.00,0.84}{aydafatehi546@gmail.com}}
\address{Masoome Zahiri, Department of Pure Mathematics, Faculty of Mathematical Sciences, Eqlid University, Eqlid, Iran}
\email{\textcolor[rgb]{0.00,0.00,0.84}{m.zahiri86@gmail.com}}
\address{Saeede Zahiri, Department of  Mathematics, Faculty of Sciences, Higher Education center of Eghlid, Eghlid, Iran}
\email{\textcolor[rgb]{0.00,0.00,0.84}{saeede.zahiri@yahoo.com}}
\thanks{*Corresponding author}


\begin{abstract} In this note, we prove that any Jordan derivation on the generalized matrix ring $T_n(R,M)$ is a derivation. This extends some well-known results of this branch due to Bre\v{s}ar et al. in the cited literature.
\end{abstract}

\maketitle

\section{\textbf{Introduction and Known Facts}}

Let $A$ be an (possibly associative) algebra and let $R$ be a commutative ring with identity. Assume that $A$ is an $R$-bi-module. Standardly, if a linear map satisfies the classical Leibniz rule, it is referred to as a \emph{derivation}:
\[
D(ab) = D(a)b + aD(b), \quad \text{for all} a, b \in A.
\]
Similarly, another version of the Leibniz rule is given by
\[
D(ab) = D(b)a + bD(a), \quad \text{for all } a, b \in A.
\]

As usual, a \emph{Jordan derivation} is an $R$-linear map that satisfies the following condition:
\[
J(a^2) = J(a)a + aJ(a), \quad \text{for all } a \in A.
\]

The traditional outcome in this regard is due to Herstein \cite{Herstein}, who showed that all Jordan derivations on a $2$-torsion free prime ring are derivations. This conclusion was later extended to semi-prime rings by Brešar \cite{Bresar}. In certain different contexts, more generalizations have been obtained as well (see \cite{Zhang,Benkovic,Yu,Bresar2,Bresar3} and references therein). Moreover, Zhang \cite{ZhangNest} demonstrated that all Jordan derivations on nest algebras are inner, which is a noteworthy fact. Likewise, Jordan derivations on upper triangular matrix algebras over commutative rings were studied by Benković \cite{Benkovic}, who established that these derivations may be written as the sum of a derivation and an anti-derivation.

The class of {\it triangular algebras}, denoted by
\[
\mathcal{T} = \begin{pmatrix}
	\mathcal{A} & \mathcal{M} \\
	0 & \mathcal{B}
\end{pmatrix},
\]
where $\mathcal{A}$ and $\mathcal{B}$ are unital algebras and $\mathcal{M}$ is a unital bi-module that is faithful as both a left $\mathcal{A}$-module and a right $\mathcal{B}$-module, is the one for which Zhang and Yu \cite{Yu} illustrated that all Jordan derivations are just derivations. Both algebraic and analytical viewpoints have recently been applied to the study of Jordan derivations in relation to trivial extension algebras of the kind described in \cite{Ghahramani1,Ghahramani2,Bresar2}. Some other closely related derivations can be found in the sources \cite{B,CH,EA} too.

\medskip

Let $R:= (R_i)_{i=1}^n$ be a family of rings and let $M:=(M_{i,j} )_{1\leq i<j\leq n}$ be a family of modules such that, for each $1\leq i<j\leq n$, $M_{i,j}$ is an $(R_i, R_j )$-bi-module. Also, assume that, for every $1\leq i<j< k \leq n$, there exists an $(R_i, R_k)$-bi-module homomorphism
$$M_{i,j}\otimes _{R_j} M_{j,k} \longrightarrow M_{i,k}$$
defined multiplicatively such that $$(m_{i,j} m_{j,k})m_{k,l} = m_{i,j} (m_{j,k}m_{k,l})\ \text{ for every}\ (m_{i,j} , m_{j,k}, m_{k,l})\in
M_{i,j} \times M_{j,k} \times M_{k,l}.$$ Then, the set $$T=\left(
                               \begin{array}{ccccc}
                                 R_1 &M_{12}&M_{13}&\cdots& M_{1n} \\
                                 0 & R_2&M_{23}&\cdots&M_{2n} \\
                                   \vdots &\ddots&\ddots&\vdots  \\
                                     0 &\cdots &&0&R_n \\
                               \end{array}
                             \right)$$ consisting of all matrices $$\left(
                               \begin{array}{ccccc}
                                m_{11}&m_{12}&m_{13}&\cdots& m_{1n} \\
                                 0 & m_{22}&m_{23}&\cdots&m_{2n} \\
                                   \vdots &\ddots&\ddots&\vdots  \\
                                     0 &\cdots &&0&m_{nn} \\
                               \end{array}
                             \right),\ m_{ii}\in R_i,\ m_{ij}\in M_{ij},\ 1\leq i\leq j\leq n,$$ equipped with the usual matrix addition and multiplication, forms a ring that is called a {\it generalized (or formal)
triangular matrix ring} and is designated by $T_n(R, M )$ (for more details, the interested authors can be referred to \cite{birk}).

The leitmotif of the present work is to show that, when $M_{1n}$ is faithful as both left $R_1$-module and right $R_n$-module and all other $M_{in}$ are faithful left $R_{in}$-module $2\leq i\leq n-1$, then any Jordan derivation of $T_n(R,M)$ is simply a derivation. This generalizes some well-established results of this sort due to Bre\v{s}ar et al. in the existing bibliography.

\section{\textbf{The Results}}

Through the text, to simplify the exposition, we just set $T:=T_n(R,M)$. When we use $E_{ij}$, we mean a matrix with a $1$ in the $i$-row and $j$-column and zeros everywhere else. We also put $T_{ij}:=E_{ii}TE_{jj}$, for any $1\leq i\leq j\leq n$. So, every $A\in T$ can uniquely be written as $$A=A_{11}+\cdots+ A_{1n}+A_{22}+\cdots+A_{nn},$$ where $A_{ij}\in T_{ij}$, $1\leq i\leq j\leq n$.

\medskip

Letting $A:=\sum_{i=1}^n\sum_{j=1}^nA_{ij}$ and $B:=\sum_{i=1}^n\sum_{j=1}^nB_{ij}$, then $AB=\sum_{t=1}^n\sum_{k=1}^nC_{tk}$, where $t\leq k$ and $$C_{tk}=A_{tt}B_{tk}+A_{tt+1}B_{t+1k}+\cdots+A_{tk}B_{kk}.$$

In order to prove our main result, we need a series of technical lemmas in what follows.

\begin{lemma}\label{1} Let $D : T \longrightarrow T$ be a linear map. If  $D$ is a Jordan derivation, then
$D(E_{ii})\in T_{1i} + T_{2i}+\cdots + T_{i-1i} + T_{ii+1}+T_{ii+2} + \cdots +T_{in}$.
 \end{lemma}

 \begin{proof} Assume that $$D(E_{ii})=\left(
                               \begin{array}{ccccc}
                                 a_{11}^{ii} &m_{12}^{ii}&\cdots &  m_{1n}^{ii} \\
                                 0 & a_{22}^{ii}&m_{23}^{ii}\cdots &m_{2n}^{ii}\\
                                   \vdots &\ddots&\ddots&\vdots  \\
                                   0 &\cdots &0& a_{nn}^{ii}\\
                               \end{array}
                             \right).$$ Since, for any $1\leq i\leq n$, $D(E_{ii})=E_{ii}D(E_{ii})+E_{ii}D(E_{ii})$ holds, we subsequently calculate that
                             $$\left(
                               \begin{array}{ccccc}
                                 a_{11}^{ii} &m_{12}^{ii}&\cdots &  m_{1n}^{ii} \\
                                 0 & a_{22}^{ii}&m_{23}^{ii}\cdots &m_{2n}^{ii}\\
                                   \vdots &\ddots&\ddots&\vdots  \\
                                   0 &\cdots &0& a_{nn}^{ii}\\
                               \end{array}
                             \right)=$$$$E_{ii}\left(
                               \begin{array}{ccccc}
                                 a_{11}^{ii} &m_{12}^{ii}&\cdots &  m_{1n}^{ii} \\
                                 0 & a_{22}^{ii}&m_{23}^{ii}\cdots &m_{2n}^{ii}\\
                                   \vdots &\ddots&\ddots&\vdots  \\
                                   0 &\cdots &0& a_{nn}^{ii}\\
                               \end{array}
                             \right)+\left(
                               \begin{array}{ccccc}
                                 a_{11}^{ii} &m_{12}^{ii}&\cdots &  m_{1n}^{ii} \\
                                 0 & a_{22}^{ii}&m_{23}^{ii}\cdots &m_{2n}^{ii}\\
                                   \vdots &\ddots&\ddots&\vdots  \\
                                   0 &\cdots &0& a_{nn}^{ii}\\
                               \end{array}
                             \right)E_{ii}=$$$$\left(
                               \begin{array}{cccccccc}
                                 0 &\cdots &0&0&0&\cdots &0 \\
                                  \vdots &\ddots&\ddots&\vdots  \\
                                  0 & \cdots &0&0&0&\cdots & 0\\
                                    0 &\cdots &0&m_{ii}^{ii}&m_{ii+1}^{ii}&\cdots &m_{in}^{ii}\\
                                     0 &\cdots &0& 0&0 &\cdots&0\\
                                     \vdots &&&\vdots  & &&\vdots  \\
                                     0 &&0&0&\cdots &0&0  \\
                               \end{array}
                             \right)+\left(
                               \begin{array}{cccccccc}
                                 0 &\cdots &0&m_{1i}^{ii}&0&\cdots &0 \\
                                 0 & \cdots &0&m_{2i}^{ii}&0&\cdots & 0\\
                                   \vdots &\ddots&\ddots&\vdots  \\
                                   0 &\cdots &0& m_{ii}^{ii}&0&\cdots &0\\
                                    0 &\cdots &0&0&0& \cdots &0\\
                                     \vdots &&&\vdots  & &&\vdots  \\
                                     0 &&0&0&\cdots &&0  \\
                               \end{array}
                             \right)$$$$=\left(
                               \begin{array}{cccccccc}
                                 0 &\cdots &0&m_{1i}^{ii}&0&\cdots &0 \\
                                 0 & \cdots &0&m_{2i}^{ii}&0&\cdots & 0\\
                                   \vdots &\ddots&\ddots&\vdots  \\
                                   0 &\cdots &0& 2m_{ii}^{ii}&m_{ii+1}^{ii}&\cdots &m_{in}^{ii}\\
                                    0 &\cdots &0&0&0& \cdots &0\\
                                     \vdots &&&\vdots  & &&\vdots  \\
                                     0 &&0&0&&\cdots &0  \\
                               \end{array}
                             \right).$$ It, therefore, follows that $m_{ii}^{ii}=2m_{ii}^{ii}$ and so $m_{ii}^{ii}=0$. Thus, $$D(E_{ii})=\left(
                               \begin{array}{cccccccc}
                                 0 &\cdots &0&m_{1i}^{ii}&0&\cdots &0 \\
                                   \vdots &\ddots&\ddots&\vdots  \\
                                 0 & \cdots &0&m_{i-1i}^{ii}&0&\cdots & 0\\
                                   0 &\cdots &0& 0&m_{ii+1}^{ii}&\cdots &m_{in}^{ii}\\
                                    0 &\cdots &0&0&0& \cdots &0\\
                                     \vdots &&&\vdots  & &&\vdots  \\
                                     0 &&0&0&&\cdots &0  \\
                               \end{array}
                             \right),$$ as required and we are done.
                             \end{proof}

For every $1\leq i\leq n$, we can put $$D(E_{ii}):= m_{1i}^{ ii}+m_{2i}^{ ii} +\cdots+ m_{i-1i}^{ ii}+m_{ii+1}^{ ii}+\cdots+m_{in}^{ ii},$$ where $m^{ ii}_{tk}\in T_{tk}$, $1\leq t\leq k\leq n$. Thereby, we arrive at the following.

\begin{lemma}\label{112} Let $D : T \longrightarrow T$ be a linear map. If $D$ is a Jordan derivation, then
$D(E_{ii}R_i)\subseteq T_{1i}+\cdots T_{i-1i}+T_{ii}+T_{ii+1}+\cdots +T_{in}$, for every $1\leq i\leq n$.
 \end{lemma}

\begin{proof} As, for every $r_i\in R_i$, $E_{ii}r_i=E_{ii}r_iE_{ii}$, we get $$D(E_{ii}r_i)=D(E_{ii}r_iE_{ii})=D(E_{ii})r_iE_{ii}+E_{ii}D(r_i)E_{ii}+E_{ii}r_iD(E_{ii}).$$ Using Lemma \ref{1} and the fact that $E_{ii}D(r_i)E_{ii}\in T_{ii}$, we deduce $$D(E_{ii}r_i)\in T_{1i} +  T_{2i}+\cdots + T_{ii} +  T_{ii+1}+T_{ii+2} + \cdots  +T_{in},$$ as needed.
\end{proof}

\begin{lemma}\label{2} Let $D : T \longrightarrow T$ be a linear map. If $D$ is a Jordan derivation, then
$D(T_{ij})\subseteq T_{1j}+\cdots T_{i-1j}+T_{ij}+T_{ij+1}+\cdots +T_{in}$, for every $1\leq i< j\leq n$.
 \end{lemma}

\begin{proof} Let $m_{ij}\in T_{ij}$. Observe that we will be ready if we show that $E_{kk}D({m_{ij}})E_{tt}=0$ is true for all $1\leq k<t\leq n$ with $k\neq i$ and $t\neq j$. To that end, as $E_{kk}m_{ij}E_{tt}=0$ is valid,
we break the verification into four cases:

\medskip

\noindent\textbf{Case 1:} If $k\neq j$ and $t\neq i$, then $E_{tt}m_{ij}E_{ kk}=0$. Hence, $$E_{kk}m_{ij}E_{tt}+E_{tt}m_{ij}E_{ kk}=0.$$ If follows that $$0=D(E_{kk})m_{ij}E_{tt}+E_{kk}D(m_{ij})E_{tt}+E_{kk}m_{ij}D(E_{tt})$$
$$+D(E_{tt})m_{ij}E_{kk}+E_{tt}D(m_{ij})E_{kk}+E_{tt}m_{ij}D(E_{kk})=$$
$$E_{kk}D(m_{ij})E_{tt}+E_{tt}D(m_{ij})E_{kk}.$$ It also follows that $0=E_{kk}E_{kk}D(m_{ij})E_{tt}E_{tt}+E_{kk}E_{tt}D(m_{ij})E_{kk}E_{tt}=E_{kk}D(m_{ij})E_{tt}$, as wanted.

\medskip

\noindent\textbf{Case 2:} If $k\neq j$ and $t=i$, then $E_{tt}m_{ij}E_{ kk}=0$ and so $$E_{kk}m_{ij}E_{tt}+E_{tt}m_{ij}E_{kk}=0.$$ Thus,
$$0=D(E_{kk})m_{ij}E_{tt}+E_{kk}D(m_{ij})E_{tt}+E_{kk}m_{ij}D(E_{tt})$$
$$+D(E_{tt})m_{ij}E_{kk}+E_{tt}D(m_{ij})E_{kk}+E_{tt}m_{ij}D(E_{kk})=$$
$$E_{kk}D(m_{ij})E_{tt}+E_{tt}D(m_{ij})E_{kk}+E_{tt}m_{ij}D(E_{kk}).$$ It now follows that $$0=E_{kk}D(m_{ij})E_{tt}=E_{kk}E_{kk}D(m_{ij})E_{tt}E_{tt}+E_{kk}E_{tt}D(m_{ij})E_{kk}E_{tt}+E_{kk}E_{tt}m_{ij}D(E_{kk})E_{tt},$$   as desired.

\medskip

\noindent\textbf{Case 3:} If $k= j$ and $t\neq i$, then $E_{tt}m_{ij}E_{ kk}=0$ and thus $$E_{kk}m_{ij}E_{tt}+E_{tt}m_{ij}E_{ kk}=0.$$ So,
$$0=D(E_{kk})m_{ij}E_{tt}+E_{kk}D(m_{ij})E_{tt}+E_{kk}m_{ij}D(E_{tt})$$
$$+D(E_{tt})m_{ij}E_{kk}+E_{tt}D(m_{ij})E_{kk}+E_{tt}m_{ij}D(E_{kk})=$$
$$E_{kk}D(m_{ij})E_{tt}+D(E_{tt})m_{ij}E_{kk}+E_{tt}D(m_{ij})E_{kk}.$$ As $E_{kk}D(m_{ij})E_{tt}\in T_{kt}$ and $D(E_{tt})m_{ij}E_{kk}+E_{tt}D(m_{ij})E_{kk}\in T_{tk}$, we infer that $$0=E_{kk}D(m_{ij})E_{tt}=D(E_{tt})m_{ij}E_{kk}+E_{tt}D(m_{ij})E_{kk},$$ as asked.

\medskip

\noindent\textbf{Case 4:} If $k= j$ and $t= i$, and since $k< t$, then via the definition of a generalized triangular matrix ring, we conclude $T_{tk}=0$ and, therefore, $m_{ij}=m_{tk}\in T_{tk}=0$. Thus, $D (m_{ij})=0$.\\

So, in either cases, we find that $E_{kk}D(m_{ij})E_{tt}=0$, as pursued.
\end{proof}

Let $D: T \longrightarrow T$ be a linear map. If $D$ is a Jordan derivation, then motivated by Lemma \ref{2}, we can design $D_{ij}(A):=E_{ii}D(A)E_{jj}$, for all $1\leq i< j \leq n$. We also may put $m_{ij}:=E_{ii}m_{ij}E_{jj}\in T_{ij}$. Hence, one obtains that
$$D(m_{ij})=\sum_{k=1}^{i}D_{kj}(m_{ij})+\sum_{t=j+1}^nD_{it}(m_{ij}),\ 1\leq i< j \leq n.$$

We proceed by proving the following technicality.

\begin{lemma} Let $D: T \longrightarrow T$ be a Jordan derivation. Then, $$D\left(\left(
                               \begin{array}{ccccc}
                                 a_{11} &m_{12}&m_{13}&\cdots& m_{1n} \\
                                 0 & a_{22}&m_{23}&\cdots&m_{2n} \\
                                   \vdots &\ddots&\ddots&\vdots  \\
                                     0 &\cdots &&0&a_{nn} \\
                               \end{array}
                             \right)\right)=\left(
                               \begin{array}{ccccc}
                                d_{11} &d_{12}&d_{13}&\cdots& d_{1n} \\
                                 0 & d_{22}&d_{23}&\cdots&d_{2n} \\
                                   \vdots &\ddots&\ddots&\vdots  \\
                                     0 &\cdots &&0&d_{nn} \\
                               \end{array}
                             \right),$$ where $d_{ii}=D_{ii}(a_{ii})$ and $$d_{ij}=D_{ij}(a_{ii}) +D_{ij}(m_{ii+1}) +\cdots +D_{ij}(m_{ij})+D_{ij}(m_{i+1j})+\cdots +D_{ij}(a_{jj}),$$ $1\leq i< j\leq n$.
\end{lemma}

\begin{proof} Set $A:=\left(
                               \begin{array}{ccccc}
                                 a_{11} &m_{12}&m_{13}&\cdots& m_{1n} \\
                                 0 & a_{22}&m_{23}&\cdots&m_{2n} \\
                                   \vdots &\ddots&\ddots&\vdots  \\
                                     0 &\cdots &&0&a_{nn} \\
                               \end{array}
                             \right)$. Thus, one computes that $$D(A)=D(E_{11}a_{11}E_{11})+D(E_{11}m_{12}E_{22})+\cdots +D(E_{11}m_{1n}E_{nn})+$$$$\cdots +D(E_{n-1n-1}m_{n-1n}E_{nn})+D(E_{nn}a_{nn}E_{nn}).$$ Now, apply  Lemma \ref{112} and Lemma \ref{2} to get the claim.
\end{proof}

\begin{lemma} Let $D : T \longrightarrow T$ be a linear map. If $D$ is a Jordan derivation, then
$D_{is}(m_{ij})=m_{ij}\otimes D_{js}(E_{jj}),$ for each $j+1\leq s$ and $m_{ij}\in T_{ij}$, $i< j $.
\end{lemma}

\begin{proof} As $i\neq j$, we derive $m_{ij}=m_{ij}E_{jj}+E_{jj}m_{ij}$. So, $$D(m_{ij})=D(m_{ij})E_{jj}+m_{ij}\otimes D(E_{jj})+D(E_{jj})\otimes m_{ij}+E_{jj}D(m_{ij}).$$ Furthermore, $$D(m_{ij})=(\sum_{k=1}^{i}D_{kj}(m_{ij})+\sum_{t=j+1}^nD_{it}(m_{ij}))E_{jj}+m_{ij}\otimes (\sum_{k=1}^{j-1}D_{kj}(E_{jj}))+\sum_{t=j+1}^nD_{jt}(E_{jj}))$$$$+
(\sum_{k=1}^{j-1}D_{kj}(E_{jj})+\sum_{t=j+1}^nD_{jt}(E_{jj}))\otimes m_{ij}+E_{jj}(\sum_{k=1}^{i}D_{kj}(m_{ij})+\sum_{t=j+1}^nD_{it}(m_{ij})).$$ However, since $\sum_{t=j+1}^nD_{it}(m_{ij})\in T_{ij+1}+\cdots +T_{in}$ and $E_{jj}\in R_{jj}$, we detect that $(\sum_{t=j+1}^nc_{it}^{m_{ij}})E_{jj}=0$. But since $m_{ij}=m_{ij}E_{jj}$ and $i< j$, we have that $$m_{ij}\otimes (\sum_{k=1}^{i-1}D_{kj}(E_{jj}))=m_{ij}E_{jj}\otimes (\sum_{k=1}^{i-1}D_{kj}(E_{jj}))=m_{ij}\otimes E_{jj}(\sum_{k=1}^{i-1}D_{kj}(E_{jj}))=0.$$ Analogously, $$(\sum_{k=1}^{i-1}D_{kj}(E_{jj})+\sum_{t=j+1}^nD_{jt}(E_{jj}))\otimes m_{ij}+E_{jj}(\sum_{k=1}^{i}D_{kj}(m_{ij})+\sum_{t=j+1}^nD_{it}(m_{ij}))=0,$$ whence $$D(m_{ij})=\sum_{k=1}^{i}D_{kj}(m_{ij})E_{jj}+m_{ij}\otimes \sum_{t=j+1}^nD_{jt}(E_{jj}).$$ It, thus, follows that $D_{is}(m_{ij})=m_{ij}\otimes D_{js}(E_{jj})$, for each $j+1\leq s$ and $m_{ij}\in T_{ij}$, $i< j$, as expected.
\end{proof}

\begin{lemma} Let $D : T \longrightarrow T$ be an $R$-linear map. If $D$ is a Jordan derivation, then $D_{ij}(E_{ii})+D_{ij}(E_{jj})=0$, for any $i\neq j$.
\end{lemma}

\begin{proof} With no loss of generality, we may assume that $i< j$. As $E_{ii}E_{jj}+E_{jj}E_{ii}=0$, we find that $$D(E_{ii})E_{jj}+E_{ii}D(E_{jj})+E_{jj}D(E_{ii})+D(E_{jj})E_{ii}=0.$$ Applying now Lemma \ref{1}, we receive that $$[\sum_{k=1}^{i-1}D_{ki}(E_{ii})+\sum_{t=i+1}^nD_{it}(E_{ii})]E_{jj}+E_{ii}[\sum_{k=1}^{j-1}D_{kj}(E_{jj})+\sum_{t=j+1}^nD_{jt}(E_{jj})]+$$$$E_{jj}[\sum_{k=1}^{i-1}D_{ki}(E_{ii})+
\sum_{t=i+1}^nD_{it}(E_{ii})]+[\sum_{k=1}^{j-1}D_{kj}(E_{jj})+\sum_{t=j+1}^nD_{jt}(E_{jj})]E_{ii}=0.$$ Further, since $i< j$, we have $\sum_{k=1}^{i-1}D_{ki}(E_{ii})E_{jj}=0$, $E_{ii}\sum_{t=j+1}^nD_{jt}(E_{jj})=0$, $E_{jj}[\sum_{k=1}^{i-1}D_{ki}(E_{ii})+\sum_{t=i+1}^nD_{it}(E_{ii})]=0$, and $[\sum_{k=1}^{j-1}D_{kj}(E_{jj})+\sum_{t=j+1}^nD_{jt}(E_{jj})]E_{ii}=0$. It, consequently, follows that $$\sum_{t=i+1}^nD_{it}(E_{ii})E_{jj}+E_{ii} \sum_{k=1}^{j-1}D_{kj}(E_{jj})=0.$$ This, however, means that $D_{ij}(E_{ii})+D_{ij}(E_{jj})=0$, as stated.
\end{proof}

\begin{lemma} Let $D : T \longrightarrow T$ be a linear map. If $D$ is a Jordan derivation, then $D_{sj}(m_{ij})= D_{si}(E_{ii})\otimes m_{ij}$, for each $s\leq i-1$ and $m_{ij}\in T_{ij}$, $i< j$.
\end{lemma}

\begin{proof} As $i\neq j$, we can get $m_{ij}=E_{ii}m_{ij}+m_{ij}E_{ii}$. So, $$D(m_{ij})= D(E_{ii})\otimes m_{ij}+E_{ii}D(m_{ij})+D(m_{ij})E_{ii}+m_{ij}\otimes D(E_{ii}).$$
Besides, $$D(m_{ij})=(\sum_{k=1}^{i-1}D_{ki}(E_{ii})+\sum_{t=i+1}^nD_{it}(E_{ii}))\otimes m_{ij}+E_{ii}(\sum_{k=1}^{i}D_{kj}(m_{ij})+\sum_{t=j+1}^nD_{it}(m_{ij}))+$$
$$(\sum_{k=1}^{i}D_{kj}(m_{ij})+\sum_{t=j+1}^nD_{it}(m_{ij}))E_{ii}+m_{ij}(\sum_{k=1}^{i-1}D_{ki}(E_{ii})+\sum_{t=i+1}^nD_{it}(E_{ii}))$$ $$=(\sum_{k=1}^{i-1}D_{ki}(E_{ii}))\otimes m_{ij}+D_{ij}(m_{ij})+\sum_{t=j+1}^nD_{it}(m_{ij}).$$ Thus, $D_{sj}(m_{ij})= D_{si}(E_{ii})\otimes m_{ij}$, for each $s\leq i-1$ and $m_{ij}\in T_{ij}$, $i< j $, as claimed.
\end{proof}

As a consequence, we yield:

\begin{corollary}\label{cor} Let $D : T \longrightarrow T$ be an $R$-linear map. If $D$ is a Jordan derivation, then $$D(m_{ij})=(\sum_{k=1}^{i-1}D_{ki}(E_{ii})\otimes m_{ij})+D_{ij}(m_{ij})+(m_{ij}\otimes \sum_{t=j+1}^nD_{jt}(E_{jj})),\ \text{for all}\ i< j.$$
\end{corollary}

We are now prepared to establish the following assertion.

\begin{lemma}\label{11} Let $D : T \longrightarrow T$ be a linear map. If $D$ is a Jordan derivation, then
$$D(m_{ij}\otimes n_{jt})= D (m_{ij})\otimes n_{jt}+m_{ij}\otimes D (n_{jt}),$$ for each $1\leq i\leq j< t\leq n$ and $m_{ij}\in T_{ij}, n_{jl}\in T_{jl}$.
\end{lemma}

\begin{proof} Since $i\leq j< t$, the definition of generalized triangular matrix rings enables us that $T_{ti}=0$, which implies that $n_{jt}\otimes m_{ij}= 0$. As $D$ is a Jordan derivation, we may write $$D(m_{ij}\otimes n_{jt})=D(m_{ij}\otimes n_{jt}+n_{jt}\otimes m_{ij})=$$$$D(m_{ij})\otimes n_{jt}+m_{ij}\otimes D(n_{jt})+D(n_{jt})\otimes m_{ij}+n_{jt}\otimes D(m_{ij}).$$ However, as $$D(n_{jt})\in T_{1t}+\cdots+T_{jt}+T_{jt+1}+\cdots+ T_{jn}$$ and $i\leq j< t$, we obtain $D(n_{jt})\otimes m_{ij}=0$. In the same way, we have $n_{jt}\oplus D(m_{ij})=0$. Finally, $$D(m_{ij}\otimes n_{jt})=D(m_{ij})\otimes n_{jt}+m_{ij}\otimes D(n_{jt}),$$ as asserted.
\end{proof}

Three more valuable consequences are these:

\begin{corollary} Let $D : T \longrightarrow T$ be a linear map. If $D$ is a Jordan derivation, then
$D_{it}(m_{ij}\otimes n_{jt})= D_{ij}(m_{ij})\otimes n_{jt}+m_{ij}\otimes D_{jt}(n_{jt})$, for each $1\leq i< j< t\leq n$ and $m_{ij}\in T_{ij}, n_{jl}\in T_{jl}$.
\end{corollary}

\begin{proof}
In view of Lemma \ref{11}, we write $$D (m_{ij}\otimes n_{jt})= D (m_{ij})\otimes n_{jt}+m_{ij}\otimes D (n_{jt}),$$ for every $1\leq i< j< t\leq n$ and $m_{ij}\in T_{ij}, n_{jl}\in T_{jl}$. As $m_{ij}=E_{ii}m_{ij}E_{jj}$ and $n_{jt}=E_{jj}n_{jt}E_{tt}$, we also write $$D (m_{ij}\otimes n_{jt})=D (m_{ij})\otimes E_{jj}n_{jt}E_{tt}+E_{ii}m_{ij}E_{jj}\otimes D (n_{jt})$$$$=D (m_{ij})E_{jj}\otimes n_{jt}E_{tt}+E_{ii}m_{ij}\otimes E_{jj}D (n_{jt}).$$ Thus, $$D_{it}(m_{ij}\otimes n_{jt})=E_{ii}D (m_{ij}\otimes n_{jt})E_{tt}=$$$$E_{ii}D (m_{ij})E_{jj}\otimes n_{jt}E_{tt}+E_{ii}m_{ij}\otimes E_{jj}D (n_{jt})E_{tt}$$$$=D_{ij} (m_{ij})\otimes n_{jt}+ m_{ij}\otimes  D _{jt}(n_{jt}),$$ concluding the arguments.
\end{proof}

\begin{corollary} Let $D : T \longrightarrow T$ be a linear map. If $D$ is a Jordan derivation, then
$D_{st}(m_{ij}\otimes n_{jt})= D_{sj}(m_{ij})\otimes n_{jt}$, for each $1\leq s< i< j< t\leq n$ and $m_{ij}\in T_{ij}, n_{jl}\in T_{jl}$.
\end{corollary}

\begin{corollary} Let $D : T \longrightarrow T$ be a linear map. If $D$ is a Jordan derivation, then
$D_{ik}(m_{ij}\otimes n_{jt})=  m_{ij}\otimes D_{jk}(n_{jt})$, for each $1\leq i< j< t< k\leq n$ and $m_{ij}\in T_{ij}, n_{jl}\in T_{jl}$.
\end{corollary}

We now continue with the following helpful statements.

\begin{lemma}\label{12} Let $D : T \longrightarrow T$ be a linear map. If $D$ is a Jordan derivation, then
$D_{ii+k}(m_{i+ti+k})\otimes m_{i+kj}=D_{ij}( m_{i+ti+k}\otimes m_{i+kj})$, for every $1\leq i <i+t\leq i+k< j \leq n$ and $m_{i+ti+k}\in T_{i+ti+k}, m_{i+kj}\in T_{i+kj}$.
\end{lemma}

\begin{proof} Since $j> i+t$, it must be that $m_{i+kj}\otimes m_{i+ti+k}=0$, and so $$D(m_{i+ti+k}\otimes m_{i+kj})=D(m_{i+ti+k}\otimes m_{i+kj}+m_{i+kj}\otimes m_{i+ti+k})=$$$$D(m_{i+ti+k})\otimes m_{i+kj}+m_{i+ti+k}\otimes D(m_{i+kj})$$$$+D(m_{i+kj})\otimes m_{i+ti+k}+m_{i+kj}\otimes D(m_{i+ti+k}).$$ It follows that $$E_{ii}D(m_{i+ti+k}\otimes m_{i+kj})E_{jj}=E_{ii}D(m_{i+ti+k})\otimes m_{i+kj}E_{jj}+E_{ii}m_{i+ti+k}\otimes D(m_{i+kj})E_{jj}$$$$+E_{ii}D(m_{i+kj})\otimes m_{i+ti+k}E_{jj}+E_{ii}m_{i+kj}\otimes D(m_{i+ti+k})E_{jj}.$$ As $E_{ii}m_{i+ti+k} = m_{i+ti+k}E_{jj}=E_{ii}m_{i+kj}=0$, we have $$D_{ij}(m_{i+ti+k}\otimes m_{i+kj})=E_{ii}D(m_{i+ti+k}\otimes m_{i+kj})E_{jj}$$$$=E_{ii}D(m_{i+ti+k})\otimes m_{i+kj}E_{jj}=E_{ii}D(m_{i+ti+k})\otimes m_{i+kj}.$$

On the other hand, it is easy to see that $E_{ii}D_{st}(m_{i+ti+k})\otimes m_{i+kj}=0$, for each $s\neq i$ and $t\neq i+k$. Finally, $D_{ij}( m_{i+ti+k}\otimes m_{i+kj})=D_{ii+k}(m_{i+ti+k})\otimes m_{i+kj}$, as suspected.
\end{proof}

\begin{lemma}\label{33} Let $D : T \longrightarrow T$ be a linear map. If $D$ is a Jordan derivation, then
$D_{ij}(m_{i k}\otimes m_{ kt})=m_{i k}\otimes D_{kj}(m_{ kt})$, for every $1\leq i \leq k\leq t< j \leq n$ and $m_{ik}\in T_{ik}, m_{kt}\in T_{kt}$.
\end{lemma}

\begin{proof} Since $t> i$, it must be that $m_{kt}\otimes m_{ik}=0$, and thus $$D(m_{ik}\otimes m_{kt})=D(m_{ik}\otimes m_{kt}+m_{kt}\otimes m_{ik})=$$$$D(m_{ik})\otimes m_{kt}+m_{ik}\otimes D(m_{kt})$$$$+D(m_{kt})\otimes m_{ik}+m_{kt}\otimes D(m_{ik}).$$ It follows that $$E_{ii}D(m_{ik}\otimes m_{kt})E_{jj}=E_{ii}D(m_{ik})\otimes m_{kt}E_{jj}+E_{ii}m_{ik}\otimes D(m_{kt})E_{jj}$$$$+E_{ii}D(m_{kt})\otimes m_{ik}E_{jj}+E_{ii}m_{kt}\otimes D(m_{ik})E_{jj}=E_{ii}m_{ik}\otimes D(m_{kt})E_{jj}.$$ It is elementarily to see that $D_{ij}(m_{i k}\otimes m_{ kt})=m_{i k}\otimes D_{kj}(m_{ kt})$, for each $1\leq i \leq k\leq t< j \leq n$ and $m_{ik}\in T_{ik}, m_{kt}\in T_{kt}$, finishing the argumentation.
\end{proof}

\begin{lemma}\label{d1} Let $D : T \longrightarrow T$ be a linear map. If $D$ is a Jordan derivation, then $D(a_{ii}m_{in})=D (a_{ii})m_{in}+a_{ii}D(m_{in})$, for every $m_{in}\in M_{in}$, $1\leq i \leq n-1$.
\end{lemma}

\begin{proof} Since $m_{in}a_{ii}=0$ for $1\leq i\leq n-1$, we write $$D(a_{ii} m_{in})=D(a_{ii} m_{in}+m_{in}a_{ii})=$$$$D(a_{ii})m_{in}+a_{ii}D(m_{in})+D(m_{in})a_{ii} +m_{in}D(a_{ii}),$$ $1\leq i\leq n-1$. Further, as $D(m_{in})\in T_{1n}+\cdots+T_{in}$, we obtain $D(m_{in})a_{ii} =0$. Analogically, $m_{in}D(a_{ii})=0$. So, $$D(a_{ii} m_{in})=
D(a_{ii})m_{in}+a_{ii} D(m_{in}),\ 1\leq i\leq n-1,$$ giving the evidence.
\end{proof}

\begin{lemma}\label{d2} Let $D : T \longrightarrow T$ be a linear map. If $D$ is a Jordan derivation, then $D(m_{1n}a_{nn})=D (m_{1n})a_{nn}+m_{1n}D(a_{nn})$, for every $m_{1n}\in M_{1n}$.
\end{lemma}

\begin{proof} Since $a_{nn}m_{1n}=0$, we get $$D(m_{1n}a_{nn})=D(m_{1n}a_{nn} +a_{nn}m_{1n})=$$$$D(m_{1n})a_{nn}+m_{1n}D(a_{nn})+D(a_{nn})m_{1n} +a_{nn}D(m_{1n}).$$  But, $D(a_{nn})\in T_{1n}+\cdots+T_{nn}$, whence $D(a_{nn})m_{1n} =0$; in a way of similarity, $a_{nn}D(m_{1n})=0$. Consequently, $$D(m_{1n}a_{nn})=D (m_{1n})a_{nn}+m_{1n}D(a_{nn}),$$ ending the argument.
\end{proof}

\begin{lemma}\label{d3} Let $D : T \longrightarrow T$ be a linear map, and assume that $M_{in}$ is faithful left $R_{i}$-module for each $1\leq i \leq n-1$, and $M_{1n}$ is faithfully right $R_1$-module. If $D$ is a Jordan derivation, then $D_{ii}(a_{ii}b_{ii})=D_{i}(a_{ii})b_{ii}+a_{ii}D_{ii}(b_{ii})$, for each $a_{ii},b_{ii}\in R_i$, $1\leq i \leq n$.
\end{lemma}

\begin{proof} Thanks to Lemma \ref{d1}, $D(a_{ii}b_{ii}M_{in})=D(a_{ii}b_{ii})M_{in}+a_{ii}b_{ii}D(M_{in})$, for $1\leq i\leq n-1$. On the other side, again with the aid of Lemma \ref{d1}, we write $$D(a_{ii}b_{ii}M_{in})=
D(a_{ii})b_{ii}M_{in}+a_{ii}D(b_{ii}M_{in})=$$$$D(a_{ii})b_{ii}M_{in}+a_{ii}D(b_{ii})M_{in}+a_{ii}b_{ii}D(M_{in}),$$ for $1\leq i\leq n-1$. So, $$D(a_{ii}b_{ii})M_{in}=D(a_{ii})b_{ii}M_{in}+a_{ii}D(b_{ii})M_{in},$$ $1\leq i\leq n-1$. However, as $M_{in}$ is faithful left $R_{i}$-module, it must be that $D_{ii}(a_{ii}b_{ii})=D_{i}(a_{ii})b_{ii}+a_{ii}D_{ii}(b_{ii})$ for each $a_{ii},b_{ii}\in R_i$, $1\leq i \leq n-1$.\\
Next, assume that $a_{nn},b_{nn}\in R_n$. Employing Lemma \ref{d2}, we know that $$D(M_{1n}a_{nn}b_{nn})=D(M_{1n})a_{nn}b_{nn}+M_{1n}D(a_{nn}b_{nn}).$$ Furthermore, again Lemma \ref{d2} is applicable to extract that $$D(M_{1n}a_{nn}b_{nn})=D(M_{1n}a_{nn})b_{nn}+M_{1n}a_{nn}D(b_{nn})=$$$$D(M_{1n})a_{nn}b_{nn}+M_{1n}a_{nn}D(b_{nn})+M_{1n}a_{nn}D(b_{nn}),$$ so that $$M_{1n}D(a_{nn}b_{nn}=M_{1n}D(a_{nn})b_{nn}+M_{1n}a_{nn}D(b_{nn}).$$. But, since $M_{1n}$ is faithful right $R_n$-module, this leads to $D(a_{nn}b_{nn})= D(a_{nn})b_{nn}+ a_{nn}D(b_{nn}),$ as asked for.
\end{proof}

We now come to our pivotal instrument which automatically assures the validity of the desired principal result quoted below.

\begin{lemma}\label{major} Let $D : T \longrightarrow T$ be a linear map, and assume that $M_{in}$ is faithful left $R_{i}$-module for each $1\leq i \leq n-1$, and $M_{1n}$ is faithfully right $R_1$-module. If $D$ is a Jordan derivation, then $(D(A)B)_{ij}+(AD(B))_{ij}=D_{ij}(AB)$, for each $1\leq i< j \leq n$ and $A,B\in T$.
\end{lemma}

\begin{proof} Assume that $A=\sum_{i=1}^n\sum_{j=1}^n a_{ij} $ and $B=\sum_{i=1}^n\sum_{j=1}^n b_{ij}$, where $a_{ij}=E_{ii}a_{ij}E_{jj}$ and $b_{ij}=E_{ii}b_{ij}E_{jj}$, $1\leq i,j\leq n$. Then, $AB=\sum_{t=1}^n\sum_{k=1}^n (ab)_{tk} $, where $t\leq k$ and $(ab)_{tk}=a_{tt}b_{tk}+a_{tt+1}b_{t+1k}+\cdots+a_{tk}b_{kk}$. Note that $$(D(A)B)_{ij}=E_{ii}D(A)BE_{jj}=E_{ii}D(A)(b_{1j}+b_{2j}+\cdots+b_{jj})E_{jj}.$$ As $E_{ii}D_{ks}(A)=0$ for every $k\neq i$, we have $$(D(A)B)_{ij}= E_{ii}(D_{ii}(A)+D_{ii+1}(A)+\cdots+D_{in}(A))(b_{1j}+b_{2j}+\cdots+b_{jj})E_{jj}.$$ Since $D_{ij}=E_{ii}D_{ij}E_{jj}$, $1\leq i,j\leq n$, we write $$(D(A)B)_{ij}= $$$$D_{ii}(A)b_{ij}E_{jj}+D_{ii+1}(A)b_{i+1j}E_{jj}+\cdots +D_{ij}(A)b_{jj}E_{jj}=D_{ii}( a_{ii})b_{ij}+$$
$$D_{ii+1}(a_{ii}+a_{ii+1}+ a_{i+1i+1} )b_{i+1j} +D_{ii+2}(a_{ii}+a_{ii+1}+ a_{ii+2}+a_{i+1i+2}+a_{i+2i+2} )b_{i+2j}$$$$+\cdots+D_{ij}( a_{ii} +a_{ii+1} +\cdots+a_{ij} + a_{i+1j}  +\cdots+ a_{jj} )b_{jj} .$$ Invoking Lemma \ref{12}, we can claim that $$(D(A)B)_{ij}=D_{ii}(a_{ii})  b_{ij}+D_{ii+1}(a_{ii}+a_{ii+1} )\otimes b_{i+1j}+D_{ii+2}(a_{ii}+a_{ii+1}+a_{ii+2})\otimes b_{i+2j}+\cdots$$$$+D_{ij}(a_{ii}+a_{ii+1}+\cdots+a_{ij} )b_{jj}+D_{ij}(a_{i+1i+1}  b_{i+1j})+D_{ij}(a_{i+1i+2}\otimes b_{i+2j}+a_{i+2i+2} b_{i+2j})$$$$+\cdots +D_{ij}(a_{i+1j}b_{jj}+a_{i+2j}b_{jj}+\cdots+a_{jj}b_{jj}).$$ On the other hand, one verifies that
$$(AD(B))_{ij}=E_{ii}AD(B)E_{jj}=(a_{ii}+a_{ii+1}+\cdots+a_{in})D(B)E_{jj}.$$ Since $D_{st}(B)E_{jj}=0$ for every $t\neq j$, we deduce that $$(AD(B))_{ij}=(a_{ii}+a_{ii+1}+\cdots+a_{in})(D_{1j}(B)+D_{2j}(B)+\cdots+D_{jj}(B))E_{jj}$$$$=a_{ii}D_{ij}(B)+a_{ii+1}\otimes D_{i+1j}(B)+\cdots+a_{ij}\otimes D_{jj}(B).$$ Notice that $D_{ij}(b_{ts})\neq 0$ if we get either $t=i$ or $s=j$ for any $1\leq i\leq j\leq n$ and $1\leq t\leq s\leq n$. It thus follows that
$$(AD(B))_{ij}=a_{ii}D_{ij}(b_{ii}+b_{ii+1}+\cdots+b_{ij}+b_{i+1j}+\cdots+b_{jj})+$$$$a_{ii+1}\otimes D_{i+1j}(b_{i+1i+1}+\cdots+b_{i+1j}+b_{i+2j}+\cdots+b_{jj})$$$$+\cdots+a_{ij-1}\otimes D_{j-1j}(b_{j-1j-1}+b_{jj})+a_{ij}\otimes D_{jj}(b_{jj})$$$$=a_{ii}  D_{ij}(b_{ij})+[a_{ii}  D_{ij}(b_{i+1j})+a_{ii+1}\otimes D_{i+1j}(b_{i+1j})]$$$$+\cdots+[a_{ii}  D_{ij}(b_{jj})+a_{ii+1}\otimes D_{i+1j}(b_{jj})+\cdots+$$$$a_{ij}\otimes D_{jj}(b_{jj})]+a_{ii}[D_{ij}(b_{ii})+D_{ij}(b_{ii+1})+\cdots+D_{ij}(b_{ij-1})]+$$$$a_{ii+1}\otimes[D_{i+1j}(b_{i+1i+1})+\cdots+D_{i+1j}(b_{i+1j-1})]+\cdots+ a_{ij-1}\otimes D_{j-1j}(b_{j-1j-1}).$$ Consulting with Lemma \ref{33} and Lemma \ref{d3}, we may assert that

$$(AD(B))_{ij}=a_{ii}  D_{ij}(b_{ij})+[a_{ii}  D_{ij}(b_{i+1j})+a_{ii+1}\otimes D_{i+1j}(b_{i+1j})]$$$$+[a_{ii}  D_{ij}(b_{i+2j})+a_{ii+1}\otimes D_{i+1j}(b_{i+2j})+a_{ii+2}\otimes D_{i+2j}(b_ {i+2j})]+$$$$\cdots+[a_{ii}  D_{ij}(b_{jj})+a_{ii+1}\otimes D_{i+1j}(b_{jj})+\cdots+$$$$a_{ij}\otimes D_{jj}(b_{jj})]+D_{ij}(a_{ii} (b_{ii}+b_{ii+1}+\cdots+b_{ij-1}))+$$$$D_{ij}(a_{ii+1}\otimes(b_{i+1i+1}+\cdots+b_{i+1j-1}))+\cdots+ D_{ij}(a_{ij-1}\otimes b_{j-1j-1}).$$ So, $$(AD(B))_{ij}=a_{ii}  D_{ij}(b_{ij})+[a_{ii}  D_{ij}(b_{i+1j})+a_{ii+1}\otimes D_{i+1j}(b_{i+1j})]$$$$+[a_{ii}  D_{ij}(b_{i+2j})+a_{ii+1}\otimes D_{i+1j}(b_{i+2j})+a_{ii+2}\otimes D_{i+2j}(b_ {i+2j})]+$$$$+\cdots+[a_{ii}  D_{ij}(b_{jj})+a_{ii+1}\otimes D_{i+1j}(b_{jj})+\cdots+a_{ij}\otimes D_{jj}(b_{jj})]$$$$+D_{ij}(a_{ii}  b_{ii})+D_{ij}(a_{ii}  b_{ii+1}+a_{ii+1}\otimes b_{i+1i+1})+\cdots+$$$$D_{ij}(a_{ii} b_{ij-1}+a_{ii+1}\otimes b_{i+1j-1}+\cdots + a_{ij-1}b_{j-1j-1}).$$
Thus, $$(D(A)B)_{ij}+(AD(B))_{ij}=D_{ii}(a_{ii})b_{ij}+D_{ii+1}(a_{ii}+a_{ii+1})\otimes b_{i+1j}$$$$+D_{ii+2}(a_{ii}+a_{ii+1}+a_{ii+2}+)\otimes b_{i+2j}+\cdots$$$$+D_{ij}(a_{ii}+a_{ii+1}+\cdots+a_{ij})b_{jj}+D_{ij}(a_{i+1i+1}b_{i+1j})$$$$+\cdots +D_{ij}(a_{i+1j}b_{jj}+a_{i+2j}b_{jj}+\cdots+a_{jj}b_{jj})+$$$$a_{ii} D_{ij}(b_{ij})+[a_{ii}  D_{ij}(b_{i+1j})+a_{ii+1}\otimes D_{i+1j}(b_{i+1j})]$$$$[a_{ii}  D_{ij}(b_{i+2j})+a_{ii+1}\otimes D_{i+1j}(b_{i+2j})+a_{ii+2}\otimes D_{i+2j}(b_ {i+2j})]$$$$+\cdots+[a_{ii}  D_{ij}(b_{jj})+a_{ii+1}\otimes D_{i+1j}(b_{jj})+\cdots+a_{ij}\otimes D_{jj}(b_{jj})]$$$$+D_{ij}(a_{ii}  b_{ii})+D_{ij}(a_{ii}  b_{ii+1}+a_{ii+1}\otimes b_{i+1i+1})+\cdots+$$$$D_{ij}(a_{ii}  b_{ij-1}+a_{ii+1}\otimes b_{i+1j-1}+\cdots + a_{ij-1}b_{j-1j-1}).$$ Referring to Lemma \ref{11} and Lemma \ref{d3}, we receive that $$(D(A)B)_{ij}+(AD(B))_{ij}=D_{ij}(a_{ii}b_{ij})+D_{ij}(a_{ii}b_{i+1j}+a_{ii+1}\otimes b_{i+1j})$$$$+D_{ij}(a_{ii} b_{i+2j}+a_{ii+1}\otimes  b_{i+2j} +a_{ii+2}\otimes  b_ {i+2j}) +\cdots+$$$$D_{ij}(a_{ii}b_{jj}+a_{ii+1}b_{jj}+\cdots+a_{ij}b_{jj})+D_{ij}(a_{i+1i+1}b_{i+1j})$$$$+\cdots +D_{ij}(a_{i+1j}b_{jj}+a_{i+2j}b_{jj}+\cdots+a_{jj}b_{jj})+$$$$D_{ij}(a_{ii}  b_{ii})+D_{ij}(a_{ii}  b_{ii+1}+a_{ii+1}b_{i+1i+1})+\cdots+$$$$D_{ij}(a_{ii}  b_{ij-1}+a_{ii+1}\otimes b_{i+1j-1}+\cdots + a_{ij-1}b_{j-1j-1}).$$ Therefore, $$(D(A)B)_{ij}+(AD(B))_{ij}=D_{ij}(a_{ii}b_{ij})+D_{ij}( a_{ii+1}\otimes b_{i+1j})+\cdots+D_{ij}( a_{ij}b_{jj})$$$$+D_{ij}(a_{i+1i+1}b_{i+1j})+D_{ij}(a_{i+1i+2}\otimes b_{i+2j}+a_{i+2i+2}b_{i+2j})+\cdots$$$$ +D_{ij}(a_{i+1j}b_{jj}+a_{i+2j}b_{jj}+\cdots+a_{jj}b_{jj})+$$$$D_{ij}(a_{ii}  b_{ii})+D_{ij}(a_{ii}  b_{ii+1}+a_{ii+1} b_{i+1i+1})+\cdots+$$$$D_{ij}(a_{ii}  b_{ij-1}+a_{ii+1}\otimes b_{i+1j-1}+\cdots + a_{ij-1}b_{j-1j-1})=$$$$D_{ij}(a_{ii}b_{ij}+ a_{ii+1}b_{i+1j}+\cdots+ a_{ij}b_{jj})+D_{ij}(a_{i+1i+1}b_{i+1j})+$$$$D_{ij}(a_{i+1i+2}\otimes b_{i+2j}+a_{i+2i+2}b_{i+2j})+\cdots$$$$ +D_{ij}(a_{i+1j}b_{jj}+a_{i+2j}b_{jj}+\cdots+a_{jj}b_{jj})+$$$$D_{ij}(a_{ii}  b_{ii})+D_{ij}(a_{ii}  b_{ii+1}+a_{ii+1}  b_{i+1i+1})+\cdots+$$$$D_{ij}(a_{ii}  b_{ij-1}+a_{ii+1}\otimes b_{i+1j-1}+\cdots + a_{ij-1}b_{j-1j-1})=$$$$D_{ij}(a_{ii}b_{ij}+ a_{ii+1}b_{i+1j}+\cdots+ a_{ij}b_{jj})+$$$$D_{ij}(a_{i+1i+1}b_{i+1j}+a_{i+1i+2}b_{i+2j}+\cdots+a_{i+1j}b_{jj})+$$$$D_{ij}(a_{i+2i+2}b_{i+2j}+a_{i+2i+3}b_{i+3j}+\cdots+a_{i+2j}b_{jj})+\cdots$$$$ +D_{ij}(a_{jj}b_{jj})+ D_{ij}(a_{ii}  b_{ii})+D_{ij}(a_{ii}  b_{ii+1}+a_{ii+1}  b_{i+1i+1})$$$$+\cdots+D_{ij}(a_{ii}  b_{ij-1}+a_{ii+1}\otimes b_{i+1j-1}+\cdots + a_{ij-1}b_{j-1j-1})=$$$$D_{ij}((AB)_{ij})+D_{ij}((AB)_{i+1j})+\cdots +D_{ij}((AB)_{jj})+D_{ij}((AB)_{ii})+$$$$D_{ij}((AB)_{ii+1})+\cdots+D_{ij}((AB)_{ij-1}).$$ Finally, as $D_{ij}((AB)_{st})=0$, where both $s\neq i$ and $t\neq j$ hold, we conclude $(D(A)B)_{ij}+(AD(B))_{ij}=D_{ij}(AB)$, as wanted.
\end{proof}

The above Lemma~\ref{major} immediately ensures the truthfulness of our major result.

\begin{theorem} Let $D : T \longrightarrow T$ be a linear map, and assume that $M_{in}$ is faithful left $R_{i}$-module for every $1\leq i \leq n-1$, and $M_{1n}$ is faithfully right $R_1$-module. Then, $D$ is a Jordan derivation if, and only if, it is a derivation.
\end{theorem}

\bibliographystyle{amsplain}

\end{document}